\documentclass{amsart}
\usepackage{amsfonts}
\usepackage{amssymb}
\usepackage{amsmath}
\usepackage{graphicx}
\newtheorem{theorem}{Theorem}

\newtheorem{corollary}{Corollary}

\newtheorem{definition}{Definition}

\newtheorem{lemma}{Lemma}
\newtheorem{proposition}{Proposition}
\newtheorem{remark}{Remark}

\begin{document}
\title[RLWS in $\Bbb{S}^{3}$]{Rotational Linear Weingarten Surfaces\\
into the Euclidean sphere}

\author[A. Barros]{A. Barros$^{1}$}
\author[J. Silva]{J. Silva$^{2}$}
\author[P. Sousa]{P. Sousa$^{3}$}




\begin{abstract}

The aim of this paper is to present a complete description of all rotational linear
Weingarten surface into the  Euclidean sphere $\Bbb{S}^{3}$. These surfaces are
characterized by a linear relation $aH+bK=c$, where $H$ and $K$ stand for their
mean and Gaussian curvatures, respectively, whereas $a,b$ and $c$ are real constants.

\medskip

\noindent\textit{$\bf{Key \ words:}$} Rotational surfaces, Linear Weingarten
surfaces.
\vspace{.2cm}

\noindent\textit{$\bf{MSC:}$} 53A10, 53C45.
\medskip

\begin{center}
Universidade Federal do Cear\'{a}, Departamento de
Matem\'{a}tica, 60455-760, Pici - Fortaleza - CE, Brazil\\
$^{1}$email: abdenago@pq.cnpq.br
\end{center}
\begin{center}
Universidade Federal do Piau\'{i}, Departamento de
Matem\'{a}tica, 64049-550, Ininga - Teresina - PI, Brazil\\
$^{2}$email: jsilva@ufpi.edu.br\\
$^{3}$email: paulosousa@ufpi.edu.br
\end{center}

\end{abstract}

\maketitle

\section{Introduction and statement of results}

The study of Weingarten's surface $M^{2}$ into the Euclidean space $\Bbb{R}^3$
remount to  classical works development around the middle of nineteenth  century
by Weingarten contained in the papers \cite{Weingarten1} and \cite{Weingarten2}.
Essentially these surfaces are a natural generalization of one with constant
curvature, more precisely,  they satisfy a relation $W(k_{1},k_{2})=0$, where
$k_{1}$ and $k_{2}$ stand for the principal curvatures of the surface while $W$
is a smooth function defined over the Euclidean space $\Bbb{R}^2$, distinguishing
when $W(k_{1},k_{2})=f(H^2-K),$ where $H$ and $K$ denote, respectively, the mean
and the Gaussian curvatures of $M^2$. We point out that replacing the Euclidean space
$\Bbb{R}^3$ either by the Euclidean sphere $\Bbb{S}^{3}$ or by the hyperbolic space
$\Bbb{H}^{3}(-1)$ we have the same definition. In the later case the work due to
Bryant \cite{Bryant} when $f(H^2-K)=c,$ for a constant $c$, retake this subject
after a long delay,  as well as works due to Rosenberg and Sa Earp \cite{SaEarp}.
When the ambient space is the Euclidean sphere $\Bbb{S}^{3}$ the case of rotational
surfaces was described by Dajczer and do Carmo \cite{cd} only for constant mean
curvature. In recent works Almeida et al. \cite{Sebastiao} and Li et al. \cite{LiWei}
obtained some results for Weingarten surfaces of the Euclidean three sphere. For
the special case $U(H,K)=0$, where $U$ is an affine function, Lopez \cite{Lopez}
described such surfaces into the Euclidean space $\Bbb{R}^{3}$ with an additional
requirement on the discriminant of $U.$  Our purpose here is to extend this later
description for a class of rotational surfaces  into the Euclidean sphere $\Bbb{S}^{3}.$
Indeed, we shall give a special attention for such surfaces satisfying  $U(H,K)=0$,
where the function satisfies
\begin{equation}
\label{LWsurfaces}
U(H,K)=aH+bK-c,
\end{equation}being $a,b,c\in\Bbb{R}$. Let us call such class of surfaces as Rotational
Linear Weingarten Surfaces or shortly by RLWS.

We can assume, without loss of generality, that $c\geq0$. Moreover, we choose
$a\not=0$ and $b\not=0$, since the cases $a=0$ and $b=0$ were analyzed by Palmas
in \cite{Palmas1} and \cite{Palmas2}. One fundamental ingredient to understand the
behavior of a RLWS as well as its qualitative properties is the sign of the its
discriminant which is defined according to $\Delta=a^{2}+4bc$. In the quoted paper
L\'{o}pez \cite{Lopez} described RLWS of hyperbolic type $(\Delta<0)$ in the Euclidean
space $\Bbb{R}^{3}$ under a suitable assumption.

Following Dajczer and do Carmo \cite{cd} we  shall use the terminology  of rotational
surface into $\Bbb{S}^{3}$ as a surface invariant by the orthogonal group $O(2)$
consider as a subgroup of the isometries group of $\Bbb{S}^{3}$. Hence we can consider
a profile curve $\gamma$ to describe the desired surface. Initially let us parametrize
the profile curve $\gamma$ in $\Bbb{S}^{2}$ by $\gamma(s)=\big(x(s),y(s),z(s)\big)$,
with $x(s)\ge0.$ If we choose $\varphi(t)=(\cos t,\sin t)$ as an element in $O(2)$
the rotational surface generated by $\gamma$ is parametrized as follows
\[
\psi:M^{2}\hookrightarrow\Bbb{S}^{3}\subset\Bbb{R}^{4}
\]
\[
(s,t)\mapsto(x(s)\cos t,x(s)\sin t,y(s),z(s)).
\]

Moreover,  we can choose the parameter $s$ to be the arc length of $\gamma.$ Then
using this parameter we obtain
\[
x^{2}(s)+y^{2}(s)+z^{2}(s)=1 \; , \; \;
\dot{x}^{2}(s)+\dot{y}^{2}(s)+\dot{z}^{2}(s)=1.
\]

In order to compute the principal curvatures of a rotational surface  $M^{2}\subset
\Bbb{S}^{3}$ we remember a fundamental lemma due to Dajczer and do Carmo \cite{cd}.

\begin{lemma}[Dajczer-do Carmo]\label{PC}
Let $M^{2}$ be a rotational surface of $\Bbb{S}^{3}$ under the above choices. Then
its principal curvatures $k_{1}$ and $k_{2}$ are given by
\[
k_{1}=-\frac{\sqrt{1-x^{2}-\dot{x}^{2}}}{x} \; \;   and \; \;
k_{2}=\frac{\ddot{x}+x}{\sqrt{1-x^{2}-\dot{x}^{2}}}.
\]
\end{lemma}

With this setting we present the fundamental relation which characterizes a RLWS
in the Euclidean sphere  $\Bbb{S}^{3}$:
\begin{equation}\label{eq2}
\frac{a}{2}x\sqrt{1-x^{2}-\dot{x}^{2}}+\frac{b}{2}(x^{2}+\dot{x}^{2})
+\frac{c}{2}x^{2}=\alpha,
\end{equation}where $\alpha$ is a constant.

Let us denote by $M_{\alpha}$ the RLWS associated with the function $x$, solution
of the equation (\ref{eq2}) and the parameter $\alpha.$ Moreover, let us consider
the special value $\alpha_{0}=\displaystyle\frac{\sqrt{a^{2}+(b+c)^{2}}}{4}
+\displaystyle\frac{b+c}{4}.$

\begin{theorem}\label{RLWsurfaces}
Let  $M_{\alpha}$ be a RLWS with $a>0$ and $\Delta\not=0$. Then we have:
\begin{enumerate}
\item[1.] $\alpha\in[\min\{0,\frac{b}{2}\},\alpha_{0}]$;

\item[2.] There are no complete immersed RLWS $M_{\alpha}\subset\Bbb{S}^{3}$
that such
\[
\alpha\in\left(\min\{0,\frac{b}{2}\},\max\{0,\displaystyle\frac{b}{2}\}\right)
\cup\left(\frac{b}{2},\frac{b+c}{2}\right);
\]

\item[3.] For
any $\alpha\in(\max\{0,\frac{b+c}{2}\},\alpha_{0}),\,\,M_{\alpha}$
is a complete immersed RLWS in $\Bbb{S}^{3}$;

\item[4.] There is only one complete immersed RLWS (Clifford torus) in
$\Bbb{S}^{3}$ that such $\alpha=\alpha_{0}$.
\end{enumerate}
\end{theorem}

Finally we prove the the following result.
\begin{theorem}\label{family}
There is a family of complete immersed RLWS in $\Bbb{S}^{3}$ that does
not contain isoparametric surfaces.
\end{theorem}

\section{Preliminaries and basic results}
From now on we shall choose the discriminant $\Delta \ne 0$ and $a> 0$. An
analogous analysis can be made for the case $a<0$. First of all we begin this
section by proving a lemma that establishes the fundamental relation (\ref{eq2}).

\begin{lemma}\label{EqIntegral1}
A surface $M^{2}\subset\Bbb{S}^{3}$ is RLWS if, and only if, the function $x$
satisfies the following differential equation:
\begin{equation*}
\frac{a}{2}x\sqrt{1-x^{2}-\dot{x}^{2}}+\frac{b}{2}(x^{2}+\dot{x}^{2})
+\frac{c}{2}x^{2}=\alpha,
\end{equation*}where $\alpha$ is a constant.
\end{lemma}
\begin{proof}Taking into account that $aH+bK=c$ we use Lemma \ref{PC} to arrive at
\begin{equation}\label{EqRLW}
\frac{a}{2}\left(\frac{\ddot{x}+x}{\sqrt{1-x^{2}-\dot{x}^{2}}}
-\frac{\sqrt{1-x^{2}-\dot{x}^{2}}}{x}\right)-b\cdot\frac{\ddot{x}+x}{x}=c.
\end{equation}

\noindent Now, note that
\[
-\frac{d}{ds}\left(\frac{a}{2}x\sqrt{1-x^{2}-\dot{x}^{2}}+\frac{b}{2}
(x^{2}+\dot{x}^{2})\right)=
\]
\[
x\dot{x}\left[\frac{a}{2}\left(\frac{\ddot{x}+x}{\sqrt{1-x^{2}-\dot{x}^{2}}}
-\frac{\sqrt{1-x^{2}-\dot{x}^{2}}}{x}\right)-b\cdot\frac{\ddot{x}+x}{x}\right].
\]

\noindent Therefore, the function $x$ satisfies the equation (\ref{EqRLW}) if, and
only if,
\[
\frac{a}{2}x\sqrt{1-x^{2}-\dot{x}^{2}}+\frac{b}{2}(x^{2}+\dot{x}^{2})
+\frac{c}{2}x^{2}=\alpha,
\] where $\alpha\in\Bbb{R}$ finishing the proof of the lemma.
\end{proof}

\begin{definition}
A solution of (\ref{eq2}) is complete if either $x$ is defined for all
$s\in \Bbb{R}$ or if the pair $(x,\dot{x})$  admits only $(0,\pm1)$ as limit values.
\end{definition}

When $(x,\dot{x})$ has $(0,1)$ or $(0,-1)$ as limit value, we deduce that the
profile curve meets orthogonally the axis of rotation. Therefore, complete
solutions of the equation (\ref{EqIntegral1}) give rise to a complete RLWS.

In order to describe the behavior of a solution of equation (\ref{eq2}) we
follow the techniques contained in the next paper  \cite{Palmas1} due to Palmas.
Initially we note that a local solution $x$ of the equation (\ref{eq2}) paired
with its first derivative  $(x,\dot{x})$, is contained on a level curve of the
function $F:D\to\Bbb{R}$ defined by
\[
F(u,v)=\frac{a}{2}u\sqrt{1-u^{2}-v^{2}}+\frac{b}{2}(u^{2}+v^{2})+\frac{c}{2}u^{2},
\]

\noindent where $D=\{(u,v)\in\Bbb{R}^{2}:u\geq0 \; {\rm and} \; u^{2}+v^{2}\leq1\}$.

\begin{lemma}\label{CriticalPoint}
Let $\mathcal{P}:=\{(u,v)\in int(D)
:\frac{\partial F}{\partial u}(u,v)=\frac{\partial F}{\partial v}(u,v)=0\}$ be the
set of critical points of $F$ contained in the interior of $D$. Then we have:
\begin{enumerate}
\item[(i)] $\mathcal{P}=\{(u_{+},0)\}\Leftrightarrow b+c\geq0$;
\item[(ii)] $\mathcal{P}=\{(u_{-},0)\}\Leftrightarrow b+c\leq0$,
\end{enumerate}
\noindent where $u_{\pm}^{2}=\frac{1}{2}\Big(1\pm
\sqrt{\frac{(b+c)^{2}}{a^{2}+(b+c)^{2}}}\Big)$.
\end{lemma}
\begin{proof} Straightforward calculations yield
\begin{eqnarray*}
\frac{\partial F}{\partial u}&=&\frac{a}{2}\sqrt{1-u^{2}-v^{2}}-a\frac{u^{2}}
{2\sqrt{1-u^{2}-v^{2}}}+(b+c)u;\\
\frac{\partial F}{\partial v}&=&-a\frac{uv}{2\sqrt{1-u^{2}-v^{2}}}+bv=\left(
-a\frac{u}{2\sqrt{1-u^{2}-v^{2}}}+b\right)v.
\end{eqnarray*}

For $(u,v)\in \mathcal{P}$ we affirm that
$-a\displaystyle\frac{u}{2\sqrt{1-u^{2}-v^{2}}}
+b\not=0$. Otherwise from $\displaystyle\frac{\partial F}{\partial u}=0$ we have
\[
\frac{a}{2}\sqrt{1-u^{2}-v^{2}}+cu=0.
\]

Hence we conclude that $(a^{2}+4bc)u=\Delta\cdot u=0$. Since $\Delta \ne 0 $ and $(u,v)\in
int(D)$ we arrive at a contradiction. Therefore, $v=0$ and
\[
\frac{a}{2}\sqrt{1-u^{2}}-\frac{a}{2}\frac{u^{2}}{\sqrt{1-u^{2}}}+(b+c)u=0.
\]
This is equivalent to
\begin{equation}\label{EqPointCritical}
 a(1-2u^{2})=-2(b+c)\sqrt{1-u^{2}}.
\end{equation}

Moreover, the solutions of the equation (\ref{EqPointCritical}) are also solutions
of the equation below
\begin{equation}\label{u2=}
u^{4}-u^{2}+\frac{a^{2}}{4[a^{2}+(b+c)^{2}]}=0.
\end{equation}

The solutions of equation (\ref{u2=}) are $u_{\pm}^{2}=\frac{1}{2}\Big(1\pm
\sqrt{\frac{(b+c)^{2}}{a^{2}+(b+c)^{2}}}\Big)$. Taking into account that $\frac{1}{u_{+}}
\cdot\frac{\partial F}{\partial u}(u_{+},0)=(b+c)-|b+c|$ and $\frac{1}{u_{-}}\cdot
\frac{\partial F}{\partial u}(u_{-},0)=(b+c)+|b+c|$, we conclude
\begin{itemize}
\item $\frac{\partial F}{\partial u}(u_{+},0)=0\Leftrightarrow
b+c\geq0$;
\item $\frac{\partial F}{\partial u}(u_{-},0)=0\Leftrightarrow
b+c\leq0$.
\end{itemize}

\noindent This completes the proof of the lemma.
\end{proof}

In what follows, let us denote by $C_{\alpha}=\{(u,v)\in D:F(u,v)=\alpha\}$ the level
curves of the function $F$ as well as $\alpha_{\pm}:=F(u_{\pm},0)$. The next lemma
enables us to determine the minimum level as well as the maximum level of $F$.

\begin{lemma}\label{MinMax}
Under the previous assumptions the following results hold:
\begin{enumerate}
\item[(i)] If $b+c\leq0$, then $\alpha\in[\frac{b}{2},\alpha_{0}]$ and $F^{-1}(\alpha_{0})
=\{(u_{-},0)\}$;
\item[(ii)] If $b+c\geq0$, then $\alpha\in[\min\{0,\frac{b}{2}\},\alpha_{0}]$ and
$F^{-1}(\alpha_{0})=\{(u_{+},0)\}$.
\end{enumerate}
\end{lemma}
\begin{proof}We start analyzing the function $F$ on the sets $X=D\cap\{u=0\}$ and
$Y=D\cap\Bbb{S}^{1}$. On the former case we have $F(u,v)=\frac{b}{2}v^{2}$ while
on the later one $F(u,v)=\frac{b}{2} +\frac{c}{2}u^{2}.$ Now, if $b+c\leq0$ we get
$b<0\leq c$, then $\displaystyle\min_{\partial D}F=\frac{b}{2}$ and
$\displaystyle\max_{\partial D}F=0<\alpha_{-}=\alpha_{0}$. Therefore
$\displaystyle\min_{D}F=\frac{b}{2}$,  $\displaystyle\max_{D}F=\alpha_{0}$ and
$F^{-1}(\alpha_{0})=\{(u_{-},0)\}$ because $(u_{-},0)$ is the only critical point of
$F$ in $int(D)$. Now, if $b+c\geq0$ Lemma \ref{CriticalPoint} yields that $(u_{+},0)$
is the only critical point of $F$ in $int(D)$. Thereby, we have two possibilities to consider:
\begin{itemize}
\item $b<0\leq c$. In this case, since $\displaystyle\min_{\partial D}F=\frac{b}{2}$,
$\displaystyle\max_{\partial D}F=\frac{b+c}{2}<\alpha_{+}=\alpha_{0},$ we get
$\displaystyle\min_{D}F=\frac{b}{2}$,  $\displaystyle\max_{D}F
=\alpha_{0}$ and $F^{-1}(\alpha_{0})=\{(u_{+},0)\}$.

\item $b>0$ and $c\geq0$. It is easy to see that, $\displaystyle\min_{\partial D}F=0$
and $\displaystyle\max_{\partial D}F=\frac{b+c}{2}<\alpha_{+}=\alpha_{0}$.
Therefore $\displaystyle\min_{D}F=0$,  $\displaystyle\max_{D}F=\alpha_{0}$ and
$F^{-1}(\alpha_{0})=\{(u_{+},0)\}$.
\end{itemize}

\noindent Thus we conclude the proof of the lemma.
\end{proof}

\begin{lemma}\label{C}
The partial derivative $\displaystyle\frac{\partial F}{\partial u}$ vanishes on the set
\begin{equation}
\Gamma =\{(u,v) \in int(D): 1-u^{2}-v^{2}=\frac{\tau^{2}}{a^{2}} u^{2}\},
\end{equation}where
$\tau=\sqrt{a^{2}+(b+c)^{2}}-(b+c).$
\end{lemma}
\begin{proof}From the expression of the partial derivatives found in the proof the Lemma
\ref{CriticalPoint} we deduce that $\displaystyle\frac{\partial F}{\partial u}=0$ if, and
only if,
\begin{equation}\label{EqFu=0}
\frac{a}{2}\sqrt{1-u^{2}-v^{2}}-a\frac{u^{2}}
{2\sqrt{1-u^{2}-v^{2}}}+(b+c)u=0.
\end{equation}

We can suppose that $u\not=0$, since $(u,v)\in int(D)$. Then $(u,v)$ satisfies the
relation (\ref{EqFu=0}) if, and only if,
\begin{equation}\label{eqn2dg}
at^{2}+2(b+c)t-a=0,
\end{equation}where $t=\displaystyle\frac{\sqrt{1-u^{2}-v^{2}}}{u}$. Since its roots are
$t_{\pm}=\frac{-(b+c)\pm\sqrt{a^{2}+(b+c)^{2}}}{a}$ and $t_{-}<0$ we deduce
$t_{+}=\displaystyle\frac{\sqrt{1-u^{2}-v^{2}}}{u}=\displaystyle\frac{\tau}{a}$
which is equivalent to $(u,v)\in\Gamma$.
\end{proof}

Geometrically, the points of the curve $\Gamma$ are the points where the level
curves have tangent vector parallel to the axis $u$.

\begin{remark}\label{uC}
Analyzing the cases $b+c\leq0$ and $b+c\geq0$ we conclude that: $b+c\leq0\Rightarrow
(u_{-},0)\in\Gamma$ whereas  $b+c\geq0\Rightarrow(u_{+},0)\in\Gamma$.
\end{remark}

\begin{lemma}\label{CinterC}
Under the previous notations the items below are valid.
\begin{enumerate}
\item[(i)] $C_{\alpha}\cap\Gamma\not=\varnothing\Leftrightarrow\frac{b}{2}<\alpha\leq\alpha_{0}$.
Moreover, if $\alpha\in(\frac{b}{2},\alpha_{0})$ then  $C_{\alpha}\cap\Gamma$
has only two elements;
\item[(ii)] $(u,v)\in C_{\alpha}\cap \{u=0\}\Leftrightarrow b\cdot v^{2}=2\alpha$;
\item[(iii)] $(u,v)\in C_{\alpha}\cap\;\Bbb{S}^{1} \Leftrightarrow c\cdot
u^{2}=2\alpha-b$.
\end{enumerate}
\end{lemma}
\begin{proof}
By Lemma \ref{C} it follows that, $(u,v)\in C_{\alpha}\cap\Gamma$ if, and only if,
\begin{equation*}
\frac{a}{2}u\sqrt{\frac{\tau^{2}}{a^{2}}u^{2}}
+\frac{b}{2}\left(1-\frac{\tau^{2}}{a^{2}}\right)u^{2}
+\frac{c}{2}u^{2}=\alpha\Leftrightarrow
\end{equation*}

\begin{equation}
\label{eqn1}
\left(\tau-\frac{b}{a^{2}}\tau^{2}
+c\right)u^{2}=2\alpha-b.
\end{equation}
Since $\tau = \frac{a^2}{\sqrt{a^{2}+(b+c)^{2}}+(b+c)}$ we deduce $\tau>\frac{b}{a^{2}}
\tau^{2}$ otherwise
\[
b\tau<\tau\big(\sqrt{a^{2}+(b+c)^{2}}+(b+c)\big)=a^2\le b\tau.
\]

Then $(u,v)\in C_{\alpha}\cap\Gamma$ if, and only if, $2\alpha>b$ which yields the first
item. While the second one is an immediate consequence of the equality $F(0,v)=\alpha$.
Now observe that $(u,v)\in C_{\alpha}\cap\;\Bbb{S}^{1}$ if, and only if,
$u^{2}+v^{2}=1$
and $F(u,v)=\alpha$. Using the function $F$ we conclude the item (iii).
\end{proof}

\section{Main result}

Next we characterize the level curves of the function $F$.

\begin{proposition}\label{LevelCurves}{\rm (Level Curves)} The level curves $C_{\alpha}$ of the
function $F$ satisfy:
\begin{enumerate}
\item If $\alpha\in\left(\min\{0,\frac{b}{2}\},\max\{0,\frac{b}{2}\}\right)$, then $C_{\alpha}$
intersects $\{(0,v):-1<v<1\}$ at two different points. Moreover, $C_{\frac{b}{2}}\cap\{u=0\}
=\{(0,\pm1)\}$, $C_{0}\cap\{u=0\}=\{(0,0)\}$, $b>0$ implies $C_{0}=\{(0,0)\}$ and $b<0$
implies $C_{\frac{b}{2}}=\{(0,\pm1)\}$;

\item If $\left(\frac{b}{2},\frac{b+c}{2}\right)$, then the level curve $C_{\alpha}$
intersects $\Bbb{S}^{1}_{+}=\{(u,v):u^{2}+v^{2}=1 \ {\rm and} \ u\geq0\}
\smallsetminus \{(0,\pm1)\}$
at two different points. Moreover, $c=0$ implies
$C_{\frac{b}{2}}=\Bbb{S}^{1}_{+}$ and $c\ne0$
implies $C_{\frac{b}{2}}\cap\;\Bbb{S}^{1}_{+}=\{(0,\pm1)\}$ and
$C_{\frac{b+c}{2}}\cap\;\Bbb{S}^{1}_{+}
=\{(1,0)\}$;

\item For any $\alpha\in\left(\max\{0,\frac{b+c}{2}\},\alpha_{0}\right)$, we get
$C_{\alpha}\cap\{u=0\}=\varnothing$ and $C_{\alpha}\cap\;\Bbb{S}^{1}_{+}
=\varnothing$;

\item If $|b+c|=\pm(b+c)$, then $C_{\alpha_{0}}=\{(u_{\pm},0)\}$.

\end{enumerate}
\end{proposition}
\begin{proof} We note that items $1,2$ and $3$ are a direct consequence of item $2$
and $3$ of Lemma \ref{CinterC}. The item (4) follows directly from Lemma \ref{MinMax},
which completes the proof of the proposition.
\end{proof}

\begin{corollary}\label{corol1}
Under the previous assumptions the following results hold:
\begin{enumerate}
 \item[(1)] If $\alpha\in\left(\min\{0,\frac{b}{2}\},\max\{0,\frac{b}{2}\}\right)
\cup\left(\frac{b}{2},\frac{b+c}{2}\right)$, then the level curve
$C_{\alpha}$ is not complete;

\item[(2)] If $\alpha\in(\max\{0,\frac{b+c}{2}\},\alpha_{0})$, then $C_{\alpha}$
is a smooth, simple closed curve.
\end{enumerate}
\end{corollary}
\begin{proof}
If $\alpha\in\left(\min\{0,\frac{b}{2}\},\max\{0,\frac{b}{2}\}\right)
\cup\left(\frac{b}{2},\frac{b+c}{2}\right)$, we get by Proposition
\ref{LevelCurves} that the level curve $C_{\alpha}$ is not defined for all
$s\in\Bbb{R}$. Therefore, $C_{\alpha}$ is not complete. This proof the item (1).
Follows directly from Proposition \ref{LevelCurves} [item (3)] that $C_{\alpha}$
is a smooth, simple closed curve.
\end{proof}

\begin{proof}[\textbf{Proof of the Theorem \ref{RLWsurfaces}}] We follow the numbering
is accordance with the statements of the theorem.
\begin{enumerate}
\item[1.] Follows directly from Lemma $4$ that $\alpha\in[\min\{0,\frac{b}{2}\},\alpha_{0}]$;
\item[2.] If the function $x$ satisfies $F(x,\dot{x})=\alpha $ and $\alpha\in\left(
\min\{0,\frac{b}{2}\},\max\{0,\frac{b}{2}\}\right)\cup\left(\frac{b}{2},\frac{
b+c}{2}\right)$, we get by Corollary \ref{corol1} that $x$ is not defined for
all $s\in\Bbb{R}$. Therefore, the RLWS associated is not complete;
\item[3.] Next we note that item $2$ of Corollary \ref{corol1} yield:
if $F(x,\dot{x})=\alpha$ and $\alpha\in(\max\{0,\frac{b+c}{2}\},\alpha_{0})$ then $x$
is defined for all $s\in\Bbb{R}$. Thereby, the RLWS associated is complete;
\item[4.] If $x$ is such that $F(x,\dot{x})=\alpha_{0}$, then $\dot{x}=0$ and $x=u_{\pm}$.
Therefore, the RLWS associated is a Clifford torus,
\end{enumerate}which completes the proof of the desired theorem.
\end{proof}

\newpage
\begin{center}
\begin{figure}
\includegraphics[scale=.35]{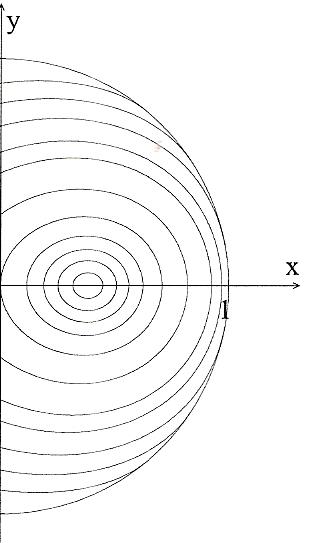}
\hspace{.4cm}
\includegraphics[scale=.35]{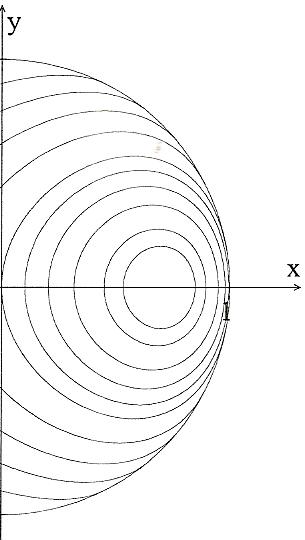}
\caption{$b+c<0$ and $b+c=0$, respectively.}
\end{figure}
\end{center}
\begin{center}
\begin{figure}
\includegraphics[scale=.35]{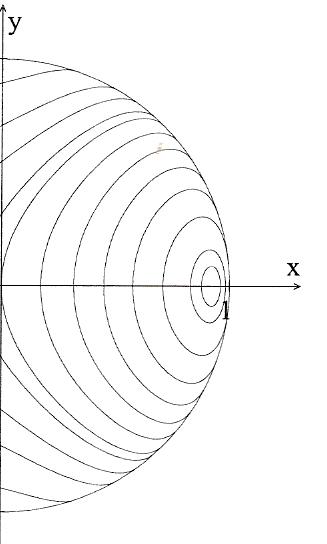}
\hspace{.4cm}
\includegraphics[scale=.35]{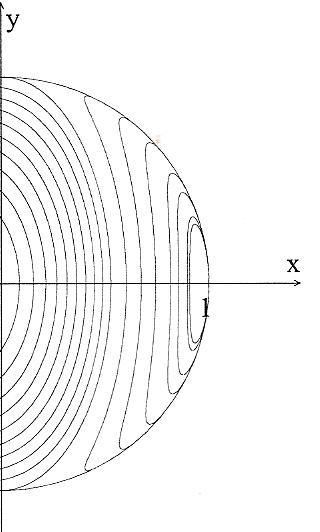}
\caption{$b+c>0$: $b<0$ and $b>0$, respectively.}
\end{figure}
\end{center}
\begin{center}
\begin{figure}
\includegraphics[scale=.35]{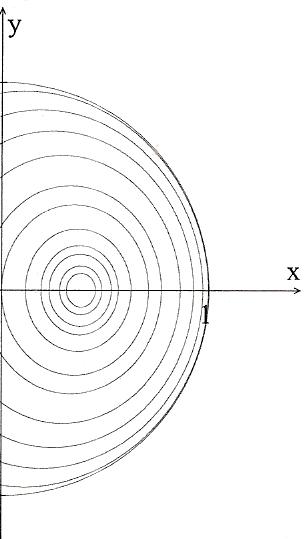}
\hspace{.4cm}
\includegraphics[scale=.35]{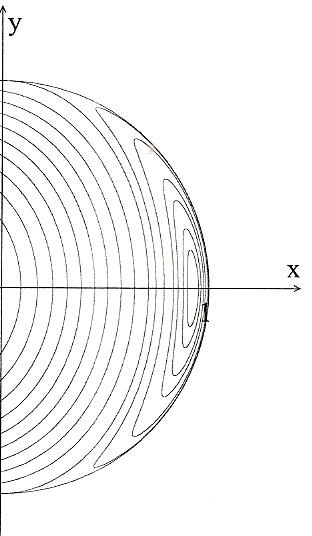}
\caption{$c=0$: $b<0$ and $b>0$, respectively.}
\end{figure}
\end{center}

In order to prove Theorem \ref{family} we shall need the following lemma.

\begin{lemma}\label{FamilyCurves}
Let $x$ be the solution of equation (\ref{eq2}) such that $x(s)\not=0$ and
$\dot{x}(s)\not=0$, $\forall\; s \in \Bbb{R}$. If $c=0$ and $k_{1}$ is constant, then
$\alpha=\frac{b}{2}$.
\end{lemma}
\begin{proof}
By Lemma \ref{PC} we get $-k_{1}x=\sqrt{1-x^{2}-\dot{x}^{2}}$ and $-k_{1}k_{2}x=x+\ddot{x}$.
Next we note that if $x$ is a solution of equation (\ref{eq2}) and $F(x,\dot{x})=\alpha$,
then
\[
ax\sqrt{1-x^{2}-\dot{x}^{2}}+b(x^{2}+\dot{x}^{2})=2\alpha.
\]

If $k_{1}=0$, we have that $x^{2}+\dot{x}^{2}=1$. It follows that $F(x,\dot{x})=\frac{b}{2}$.
Now suppose $k_{1}\not=0$. In this case, $-ak_{1}x^{2}+b(x^{2}+\dot{x}^{2})=2\alpha$.
Differentiating this equality we obtain
\[
-2ak_{1}x\dot{x}+2b(x+\ddot{x})\dot{x}=0\Leftrightarrow-2ak_{1}x+2b(x+\ddot{x})=0
\]
\[
\Leftrightarrow-a+b\frac{x+\ddot{x}}{k_{1}x}=0\Leftrightarrow-a-bk_{2}=0
\Leftrightarrow k_{2}=-\frac{a}{b}.
\]

It follows from the expression of $k_{2}$ that $\sqrt{1-x^{2}-\dot{x}^{2}}=\frac{a}{b}x+\beta$,
where $\beta\in\Bbb{R}$. Thus, $k_{1}=-\frac{a}{b}-\frac{\beta}{x}$. As $k_{1}$ is constant,
we deduce that $\beta=0$ as well as $k_{1}=k_{2}=-\frac{a}{b}$. Therefore,
\begin{eqnarray*}
F(x,\dot{x})&=&\frac{a}{2}x\sqrt{1-x^{2}-\dot{x}^{2}}+\frac{b}{2}(x^{2}+\dot{x}^{2})\\
&=&-\frac{a}{2}k_{1}x^{2}+\frac{b}{2}(1-k_{1}^{2}x^{2})=\frac{b}{2},
\end{eqnarray*}which finishes the proof of lemma.
\end{proof}

Finally we shall prove the Theorem $2.$

\begin{proof}[\textbf{Proof of the Theorem \ref{family}}]
If $x$ is solution of equation (\ref{eq2}) such that $F(x,\dot{x})=\alpha$ and
$\alpha\in(\max\{0,\frac{b}{2}\},\alpha_{0})$, it follows from Proposition \ref{LevelCurves}
that $(x,\dot{x})$ is a smooth, simple closed curve and $x(s)\not=0\; \forall\; s\in\Bbb{R}$.
Thereby, Lemma \ref{CinterC} enables us to suppose, without loss of generality, that
$\dot{x}(s)\not=0 \; \forall\; s\in\Bbb{R}$. Therefore, it follows from Lemma
\ref{FamilyCurves} that when $c=0$ the RLWS associated with $x$ is not isoparametric.
Moreover, by Theorem \ref{RLWsurfaces} we deduce that such surfaces are complete and
immersed. This completes the proof of the theorem.
\end{proof}

\end{document}